\documentclass[pdftex,11pt,a4paper,oneside]{article}
\DeclareUnicodeCharacter{221A}{$\sqrt{}$}
\DeclareUnicodeCharacter{22C9}{$\rtimes$}
\usepackage[margin=2.5cm]{geometry}
\usepackage[british]{babel}
\usepackage[utf8]{inputenc}
\usepackage{amsfonts}
\usepackage{mathtools}
\usepackage{amsthm}
\usepackage{amssymb}
\usepackage{hyperref}
\usepackage{blkarray}
\usepackage{tikz-cd}
\usepackage{caption}
\usepackage{subcaption}
\usepackage{authblk}
\usepackage{csquotes}

\theoremstyle{plain}
\newtheorem{theorem}{Theorem}[section]
\newtheorem{corollary}{Corollary}[section]
\newtheorem{lemma}{Lemma}[section]
\newtheorem{proposition}{Proposition}[section]

\newtheorem{remark}{Remark}[section]
\usepackage{enumitem}
\newtheorem{definition}{Definition}[section]

\newcommand{\rstr}[2]{\left. #1 \right\rvert_{#2}}

\newcommand{\abs}[1]{\left\lvert#1\right\rvert}

\DeclareMathOperator{\coker}{coker}
\DeclareMathOperator{\Ima}{Im}
\DeclareMathOperator{\fe}{FE}
\usepackage[
backend=biber,
style=math-numeric,
sorting=anyt,
url=false,
doi=false,
isbn=false
]{biblatex}
\title{On the Ruelle-Mayer Transfer Operators for Hölder Continuous Functions}
\author[1]{Alexander Baumgartner}
\affil[1]{{Centro di Ricerca Matematica Ennio De Giorgi, Scuola Normale Superiore, Italy} }
\date{\today}

\begin{document}

\maketitle
\abstract{We consider a family of operators connected with the geodesic flow on the modular surface. We show certain spectral information is retained after expanding their domain to the space of $\alpha$-Hölder continuous functions on the unit interval. For example, the point spectra associated with the Maass cusp forms and non-trivial zeroes of the Riemann zeta function to the right of the critical line remain unchanged when the Hölder constant is $(1/2+\varepsilon)$ and $3/4$ respectively. We briefly consider a three-term functional equation introduced by Lewis in the Hölder setting and provide a partial classification of solutions in this setting.}
\section{Introduction}

Let $G:(0,1)\to [0,1)$ be the Gauss map 
\begin{equation}\label{equation: gauss map}
    G(x) = \frac{1}{z}-\left\lfloor\frac{1}{z}\right\rfloor.
\end{equation}
We associate to $G$ the following family of transfer operators which we formally define as acting on some space of functions $f:[0,1]\to \mathbb{C}$:
\begin{equation}\label{equation: Mayer transfer operator}
    \mathcal{L}_{\beta}(f)(z)=\sum_{n=1}^\infty \frac{1}{(n+z)^{2\beta}}f\left(\frac{1}{n+z}\right),
\end{equation}
where $\beta \in \{z\in \mathbb{C}: \Re(\beta)>1/2\}$. When $\beta=1$, this is the familiar Ruelle-Perron-Frobenius operator for $G$ and can formally be interpreted as the right-adjoint of the Koopman operator $f\mapsto f\circ G$. The operators in the case $\beta \neq 1$ were first introduced by Mayer \cite{Mayer90}, who considered them acting on the Banach space of functions $A_\infty(D)$ which are holomorphic on $D:=\{z\in\mathbb{C}: \abs{z-1}< 3/2\}$ with continuous extension to $\overline{D}$.

For this choice of Banach space, Mayer also showed that the operator-valued function $\beta\mapsto \mathcal{L}_{\beta}$ admits a meromorphic continuation on $\mathbb{C}$, with simple poles at $\beta=k/2$ for $k\in\{1,0,-1,-2,\ldots\}$. The continuation for $\Re(\beta) >0$ can be obtained by rewriting $f=f(0)+f^\star$ with $f^\star:=f-f(0)$ and noting that
\begin{equation}\label{equation: Mayer transfer extended}
    \mathcal{L}_{\beta}(f)(z)=f(0)\zeta_H(2\beta,z+1)+\sum_{n=1}^\infty \frac{1}{(n+z)^{2\beta}}(f^\star)\left(\frac{1}{n+z}\right),
\end{equation}
where $\zeta_H(s,z)=\sum_{n=0}^\infty (n+z)^{-s}$ is the Hurwitz zeta function, which has an analytic continuation for all $s\neq 1$ and a simple pole of residue $1$ at $s=1$. In fact, it can further be analytically continued to $\beta \geq -k/2$ for any $k\in\mathbb{N}$ by noticing that
\begin{equation}\label{equation: Mayer transfer extended2}
    \mathcal{L}_{\beta}(f)(z)=\sum_{n=0}^k\frac{d^nf}{dz^n}(0)\frac{\zeta_H(n+2\beta,z+1)}{n!}+\sum_{n=1}^\infty \frac{1}{(n+z)^{2\beta}}(f^\star)\left(\frac{1}{n+z}\right),
\end{equation}
where $f^\star(z)=f(z)-  \sum_{n=0}^k\frac{d^nf}{d^nz}(0)z^n/n!$.
Furthermore $\mathcal{L}_{\beta}$ is a nuclear operator on $A_\infty(D)$ of order $0$ for all $\beta$. 

The Mayer transfer operator encodes geometric information about the modular surface, which is the quotient orbifold $\mathbf{M}$ obtained from considering the Poincaré upper half plane $\mathbb{H}:\{x+iy:x,y\in \mathbb{R},  y>0\}$ equipped with the usual Riemannian metric $ds^2=(dx^2+dy^2)/y^2$ and identifying points via equivalence relation $\mathbb{H}\ni z\sim (az+b)/(cz+d)$ for all $a,b,c,d\in\mathbb{Z}$ with $ad-bc=1$. Since $\mathcal{L}_{\beta}$ is a nuclear operator, the Fredholm determinants $\det(1-\mathcal{L}_{\beta})$ and $\det(1+\mathcal{L}_{\beta})$ are well-defined. These relate to Selberg Zeta function $Z(s)$  by the formula 
\begin{equation}\label{equation: zeta functions}
    Z(s)= \det(1-\mathcal{L}_s)(1+\mathcal{L}_s),
\end{equation}
see \cite{Mayer91}. 
The Selberg zeta function is closely connected with the spectral theory of the Laplace-Beltrami operator $\Delta:=-y^2(d^2/dx^2+d^2/dy^2)$ \cite{Iwaniec}. Equation \eqref{equation: Mayer transfer extended} is thus a way of connecting the spectral theory of $\Delta$ with that of $\mathcal{L}_\beta$. This connection has been elaborated upon by Lewis, Zagier, Mayer and Chang \cite{Lewis97,ZagierLewis,MayerChang96,MayerChang98}.

From a dynamical perspective, it is interesting to see to what extent the aforementioned spectral properties of $\mathcal{L}_s$ survive when we consider e.g. Hölder continuous functions instead. For $\alpha\in (0,1)$, let $C^\alpha([0,1])$ denote the space of $\alpha$-Hölder functions on $[0,1]$. If $\alpha = k+\eta$ with $k\in\mathbb{N}$ and $\eta\in [0,1)$, let $C^\alpha([0,1])$ be the set of functions for which the $k$-th derivative exists and is continuous or $\eta$-Hölder if $\eta=0$ or $\eta\in (0,1)$ respectively. Denote by $L_\beta$ the operator \eqref{equation: Mayer transfer extended2} acting on $C^\alpha([0,1])$, which we shall show to be well-defined and continuous for large enough $\alpha$. These operators are not compact, but we can decompose the spectrum $\sigma(L_\beta)$ into the essential spectrum $\sigma_e(L_\beta)$ and the discrete spectrum $\sigma_{\mathrm{disc}}(L_\beta)$ which consists of isolated eigenvalues of finite algebraic multiplicity. We prove the following.
\begin{theorem}\label{theorem: radius of essential spectrum}
    On the half-plane $\Re(\beta) >(1-\alpha)/2$, the radius $\rho_e(L_\beta):= \sup_{\lambda\in \sigma_\varepsilon(L_\beta)}\abs{\lambda}$ of the essential spectrum satisfies
    \[\rho_e(L_\beta)\leq \rho(\mathcal{L}_{\Re(\beta) +\alpha}):=\sup_{\lambda\in \sigma(\mathcal{L}_{\beta+\alpha})}\abs{\lambda} .\]
    Furthermore, the generalised eigenspaces belonging to eigenvalues $\lambda$ with $\abs{\lambda}>\rho_e(L_\beta)$ consist of analytic functions which extend to generalised eigenfunctions of $\mathcal{L}_\beta$. 
\end{theorem}
The statement in the abstract concerning Maass cusp forms and nontrivial zeroes of the Riemann zeta function, which we make more precise in Remark \ref{remark: abstract}, turns out to follow immediately from this theorem.

Comparing the behaviour of transfer operators for different levels of regularity is a somewhat common theme in dynamics. One typically expects for a transfer operator associated to a sufficiently \enquote{nice} expanding map that the essential spectral radius vanishes for the operator acting on spaces of increasingly smooth functions, see e.g. \cite{Butterley}. In a sense the above theorem is therefore not too suprising. Indeed, we shall show that if $\alpha$ is a natural number, the proof of Theorem \ref{theorem: radius of essential spectrum} is a mundane application of standard techniques.

 However, the proof of the case $1/2<\alpha < 1$ is perhaps more interesting. In that case, we still obtain uniformly bounded estimates for $\rho_e(L_\beta)$ on vertical lines $\Re(\beta)=\sigma$ even when $\sigma < 1/2$, where the function $\beta\mapsto L_\beta$ is unbounded. However, by \eqref{equation: Mayer transfer operator}, the unboundedness in $\beta$ is due to the contribution rank one operator $f\mapsto f(0)\zeta_H(2\beta,z+1)$, which does not contribute to the essential spectral radius. Nevertheless, we do not obtain this result using a Doeblin-Fortet-Lasota-Yorke inequality, so our proof is substantially different from the standard approaches to these type of problems.

In Section \ref{section: three-term functional equation}, we interpret this result in terms of Lewis' three-term functional equation. In particular, we can interpret this as an extension of the \enquote{bootstrapping} result of Lewis and Zagier, who proved in \cite{ZagierLewis} that real-analytic solutions of this equation on $(1,\infty)$ satisfying certain growth conditions automatically extend to holomorphic functions on the cut plane $\mathbb{C}\backslash(-\infty,-1]$. We show that for $\Re(\beta)>0$, we may even replace real-analyticity with a \textit{Hölder condition} and obtain that it is still the restriction of a holomorphic function on $\mathbb{C}\backslash(-\infty,-1]$.
\section{Preliminaries}

\subsection{On the Gauss Map}\label{subsection: on the Gauss map}
We briefly elucidate the Markov Structure of the Gauss map and define some notation which will be useful later on. There is a natural countable partition of the unit interval into intervals $(1/(n+1),1/n], n\in \mathbb{N}$. Each of these intervals are mapped by $G$ homeomorphically onto $[0,1)$. We denote by $\psi_n:[0,1) \to (1/(n+1),1/n]$ the inverse branch $\psi_n(x)= 1/(n+x)$. Given a sequence $n_1,n_2,\ldots,n_k\in\mathbb{N}$ we denote by $\psi_{n_1,n_2,\ldots,n_k}$ the map defined by
\begin{equation}\label{equation: inverse branches}
    \psi_{n_1,n_2,\ldots,n_k}:=\psi_{n_1}\circ \cdots \psi_{n_{k-1}}\circ\psi_{n_k}(x) = \frac{1}{ n_1+ \frac{1}{\ddots +\frac{1}{n_k+x}}}.
\end{equation}
We also denote the right-hand side by $[n_1,\ldots,n_k+x]$ as a simplifying bit of notation and we let $\langle n_1,\ldots,n_k\rangle$ denote the set $\{[n_1,\ldots,n_k+x]: x\in [0,1)\}$. 

\subsection{The Essential Spectral Radius}
In Section \ref{section: uniformly bounded essential spectrum}, we provide estimates for the essential spectrum of $L_\beta$, so we provide a brief reminder of some definitions. Let $I$ be the identity operator. We recall that the spectrum $\sigma(T)\subset \mathbb{C}$ of an operator $T$ on a Banach space $B$ is the set of all $\lambda \in\mathbb{C}$ for which $(T-\lambda I)$ is not invertible. Following Browder \cite{Browder1961}, we define the essential spectrum $\sigma_\varepsilon(T)\subset \sigma(T)$ to be the set of all $\lambda$ for which $(T-\lambda I)B$ is not closed, $\bigcup_{r=1}^{+\infty} \ker(T-\lambda I)^r$ is infinite-dimensional or for which $\lambda$ is a limit point of $\sigma(T)$. A nice consequence of this definition is that $\sigma(T)\setminus \sigma_\varepsilon(T)$ consists solely of isolated eigenvalues with finite-dimensional generalised eigenspaces.

There is a convenient formula for the essential spectral radius due to Nussbaum \cite{Nussbaum} that we now briefly introduce. Let $K$ be the subspace of compact operators on $B$. Define the seminorm $\|\cdot\|_K$ on the space of bounded linear operators on $B$ by
\[\|T\|_K = \min \{\|T-T'\|:T'\in K\},\]
and let $T:B\to B$ be a bounded linear operator. Then its essential spectral radius satisfies 
\begin{equation}\label{equation: Nussbaum}
\rho_e(T):= \sup_{\lambda\in \sigma_\varepsilon(T)}\abs{\lambda} = \lim_{l\to +\infty} \|T^l\|_K^{1/l}.    
\end{equation}

\subsection{Definition of Transfer Operators}
In this subsection, we show that the Mayer transfer operators with $\Re(\beta)>0$ are well-defined on Hölder-continuous functions with a sufficiently large Hölder exponent. We also provide a brief overview of some definitions and notations which will be used throughout this preprint.

Denote for $0<\alpha<1$ by $C^\alpha([0,1])$ the space of $\alpha$-Hölder continuous functions on $[0,1]$ equipped with the norm
\[ \|f\| := \| f \|_\infty + \|f\|_\alpha, \text{where }\| f\|_\alpha= \sup_{\substack{x,y:\\x\neq y}}\frac{\lvert f(x)-f(y)\rvert}{\abs{x-y}^\alpha}\]
and $\| f \|_\infty = \sup_{[0,1]} \abs{f}$.
It will turn out to be useful to define for $f\in C^\alpha([0,1])$ the auxiliary function $\lvert f\rvert_\alpha$ on $[0,1]$ by \[\lvert f\rvert_\alpha(x)=\sup_{\substack{y:\\x\neq y}}\frac{\lvert f(x)-f(y)\rvert}{\abs{x-y}^\alpha}.\]
We remark that $\| \lvert f\rvert_\alpha \|_\infty = \|f\|_\alpha$.

For $\Re(\beta)>1/2$, the operator \eqref{equation: Mayer transfer operator} is also well defined on the space $L^\infty([0,1])$ of bounded functions equipped with the $\|\cdot\|_\infty$-norm. We denote the operator acting on this space by $\mathfrak{L}_\beta$.

\begin{lemma}\label{lemma: mayer operator inequalities}
    Let $\alpha \in (0,1]$, $\Re(\beta) > (1-\alpha)/2$ with $\beta \neq 1/2$. Then the operator $L_\beta$ is well-defined and continuous on $C^\alpha([0,1])$. Furthermore there exists a constant $C$ and a family of constants $(D_{\varepsilon})_{\varepsilon>0}$ depending only on $\alpha$ and $\Re(\beta)$ such that
    \[\left\|L_\beta(f)\right\|_\infty \leq C\|f\|\text{, and }\]
    \[\lvert L_\beta (f)\rvert_\alpha(x)\leq  \mathfrak{L}_{\Re(\beta)+\alpha}((1+\varepsilon)|f|_\alpha )(x) + D_\varepsilon\|f\|_\infty \text{ for all } x\in [0,1], \varepsilon>0 .\]
\end{lemma}
\begin{proof}
     We note that for $f\in C^\alpha([0,1])$ we have that $\abs{f(x)-f(0)}\leq \| f\|_\alpha x^\alpha $ as $x\to 0$, hence we have that
     \begin{equation}\label{equation: holder difference}
         \abs{L_\beta(f)(z)}\leq \|f\|_\infty\abs{\zeta_H(2\beta,z+1)} + \| f\|_\alpha \sum_{n=1}^\infty \frac{1}{(n+z)^{2\Re(\beta)+\alpha}}.
     \end{equation}
To bound the auxiliary function, we start with a basic application of Taylor's theorem. Let $\theta$ be a complex number with $\Re(\theta)>0$. Let $x,y\in [0,1]$. Then 
\begin{equation}\label{equation: Taylor}
    \abs{(n+y)^{-\theta}-(n+x)^{-\theta}+\theta(n+y)^{-\theta-1}(y-x)} \leq K(y-x)^2
\end{equation}
for a constant $K>0$ which can be uniformly bounded over all $x,y$ and $n\geq 1$.

Note that
\[\frac{f\left(\frac{1}{n+x}\right)}{(n+x)^{2\beta}}-\frac{f\left(\frac{1}{n+y}\right)}{(n+y)^{2\beta}}=\left(\frac{f\left(\frac{1}{n+x}\right)}{(n+x)^{2\beta}}-\frac{f\left(\frac{1}{n+y}\right)}{(n+x)^{2\beta}}\right) + \left(\frac{f\left(\frac{1}{n+y}\right)}{(n+x)^{2\beta}}-\frac{f\left(\frac{1}{n+y}\right)}{(n+y)^{2\beta}}\right).\]
By first using the definition of the auxiliary function, and then applying \eqref{equation: Taylor} for $\theta =\alpha$ we find that there exists some constant $R>0$ which is independent of $f$ for which
\begin{equation*}
    \begin{split}
        \abs{f\left(\frac{1}{n+x}\right)-f\left(\frac{1}{n+y}\right)}&\leq 
         \lvert f\rvert_\alpha\left(\frac{1}{n+x}\right)\abs{\frac{x-y}{(n+y)(n+x)}}^\alpha  \\ &\leq\lvert f\rvert_\alpha\left(\frac{1}{n+x}\right)\frac{\abs{x-y}^\alpha}{(n+x)^{2\alpha}}\left(1+\frac{R\abs{x-y}}{n^{1+\alpha}}\right).
    \end{split}
\end{equation*}
By estimating the second term using \eqref{equation: Taylor} for $\theta = 2\beta$, we obtain that 
\begin{equation}\label{equation: summand for Hölder}
    \frac{\abs{\frac{f\left(\frac{1}{n+x}\right)}{(n+x)^{2\beta}}-\frac{f\left(\frac{1}{n+y}\right)}{(n+y)^{2\beta}}}}{\abs{x-y}^\alpha} \leq \frac{\lvert f\rvert_\alpha(\frac{1}{n+x})}{(n+x)^{2(\Re(\beta)+\alpha)}}\left(1+\frac{R\abs{x-y}}{n^{1+\alpha}}\right) + D'\abs{x-y}^{1-\alpha}\frac{\abs{f\left(\frac{1}{n+y}\right)}}{(n+y)^{2\beta +1 }} 
\end{equation}
for some uniformly bounded constant $D'>0$. 
 Let $\varepsilon>0$. If $\abs{x-y}\leq \varepsilon/R$, we sum over all $n$ in \eqref{equation: summand for Hölder} to obtain
 \[\sum_{n=1}^\infty\frac{\abs{\frac{f\left(\frac{1}{n+x}\right)}{(n+x)^{2\beta}}-\frac{f\left(\frac{1}{n+y}\right)}{(n+y)^{2\beta}}}}{\abs{x-y}^\alpha}\leq \mathfrak{L}_{\Re(\beta)+\alpha}((1+\varepsilon)|f|_\alpha )(x) + D'\|L_{\Re(\beta)+1/2}|f|\|_\infty ,\]
 where $|f|$ denotes the function $x\mapsto |f(x)|$. Else let $N\in\mathbb{N}$ be large enough such that $ R/N^{1+\alpha} \leq \varepsilon$. Sum both sides of the inequality
\begin{equation*}
   \frac{\abs{\frac{f\left(\frac{1}{n+x}\right)}{(n+x)^{2\beta}}-\frac{f\left(\frac{1}{n+y}\right)}{(n+y)^{2\beta}}}}{\abs{x-y}^\alpha} \leq \frac{2\|f\|_\infty}{n^{2\Re(\beta)}\abs{x-y}^\alpha}
\end{equation*} for $n\in\{1,2,\ldots,N\}$ and add the respective sides in \eqref{equation: summand for Hölder} over all $n> N$. After taking the supremum over all $y$ we then obtain that $\lvert L_\beta (f)\rvert_\alpha(x)\leq  \mathfrak{L}_{\Re(\beta)+\alpha}((1+\varepsilon)|f|_\alpha)(x) + E_\varepsilon\|L_{\Re(\beta)+1/2}|f|\|_\infty$ for some $E_\varepsilon>0$. The lemma follows from boundedness of the operator $L_{\Re(\beta)+1/2}$ on $L^\infty([0,1])$

\end{proof}
\begin{remark}
    To apply the usual method for obtaining quasi-compactness with essential spectral radius less than $r$, we would need an estimate of the form 
    \[ \|L_\beta^n f\|_\alpha\leq R_n \|f\|_\infty +r_n\|f\|_\alpha,\]
    where $R_n,r_n>0$ and $\liminf_n (r_n)^{1/n}\leq r$. 
    It is indeed possible to find an upper bound for $\|L_\beta^n f\|_\alpha$ by iteratively using Lemma \ref{lemma: mayer operator inequalities}. Unfortunately, this does not yield a quantitatively good upper bound, both due to the dependence on $\|f\|_\alpha$ in the estimate \eqref{equation: holder difference} and when considering $\beta$ for which $\sup_{z\in[0,1]}\abs{\zeta_H(2\beta,z+1)}$ is large. 
\end{remark}

\subsection{Operator-Valued Holomorphic Functions}
On occasion, we shall make use of the notion of holomorphy of functions, so we briefly recall the definition. We follow Chapter III, Section 1 and Section 3 in \cite{Kato}. Let $U\subset C$ be a simply connected open set of $\mathbb{C}$. If $B$ is a Banach space with norm $\|\cdot\|$, or more generally a topological vector space, and $F: U\to B: \beta\mapsto f_{\beta}$ is some $B$-valued function, we define $F$ to be homeomorphic if $(f_\beta-f_{\beta_0})/(\beta-\beta_0)$ has a limit  as $\beta\to\beta_0$ for $\beta_0\in U$.

If $B_1,B_2$ are Banach spaces, let $L(B_1,B_2)$ be the set of bounded operators equipped with the usual operator norm, then a map $\mathcal{K}:U\mapsto L(B_1,B_2):\beta\mapsto K_\beta$ is defined to be holomorphic if it is holomorphic according to the Banach space definition outlined above. 

Many of the usual theorems in complex analysis continue to hold for Banach or topological vector space-valued functions. In particular, holomorphic implies infinitely differentiable with a locally converging Taylor expansion. The unique extension theorem also continues to hold.
Finally, note that it is straightforward to show that $\beta\to L_\beta$ is a meromorphic function on $\Re(\beta)>(1-\alpha)/2$ for any $\alpha\in\mathbb{R}_{>0}$ with derivative 
\[ \left(\frac{d L_\beta}{d\beta}\right)(f)(z)= -\sum_{n=0}^k(n+2\beta)\frac{d^nf}{dz^n}(0)\frac{\zeta_H(n+2\beta+1,z+1)}{n!} -2 \sum_{n=1}^\infty \frac{\ln(n+z)}{(n+z)^{2\beta}}(f^\star)\left(\frac{1}{n+z}\right).\]
\section{Uniformly Bounded Essential Spectrum}\label{section: uniformly bounded essential spectrum}
By \eqref{equation: Nussbaum}, we can estimate the essential spectral radius by finding for each $l\in \mathbb{N}$ a compact operator $T'$ for which $\|L_\beta^l - T' \|$ is small. A first naive approach would be to consider the operator $P$ which sends $f$ to a piecwise linear approximation that interpolates $f$ on the points $0,\frac{1}{N},\frac{2}{N},\cdots,1$ and considering $\|L_\beta^l-L_\beta^l\circ P\|$ as $N\to \infty$. For $\Re(\beta)>0$, we indeed have by Lemma \ref{lemma: mayer operator inequalities} that for any $\varepsilon>0$ there exists some $N$ for which $\lvert (L_\beta-L_\beta\circ P) (f)\rvert_\alpha(x)\leq  \mathfrak{L}_{\Re(\beta)+\alpha}((1+\varepsilon)|f-P(f)|_\alpha)(x)+\varepsilon$. However, we do not have sufficient control over the $\|\cdot\|_\infty$-norm. We remedy this by considering a \textit{countable} interpolation which ensures $(L_\beta^l\circ P)f$ remains a (nonlinear) interpolation of
$L_\beta^l f$. 

To this end, let $\mathcal{P}_{0,N}$ be the set of points $\left\{0,\frac{1}{N},\frac{2}{N},\cdots,\frac{k-1}{N}\right\}$. For $l\in \mathbb{N}$, we let \[\mathcal{P}_{l,N}:=\bigcup_{n_1,\ldots,n_{l}\in \mathbb{N}}\psi_{n_1,\ldots,n_{l}}(\mathcal{P}_{0,N})\text{, and } \widetilde{\mathcal{P}_{l,N}} =\{1\}\cup\bigcup_{l'=0}^l\mathcal{P}_{l',N},\]
i.e. $\mathcal{P}_{l',N}$ is the set of all points of the form $[n_1,\ldots,n_{l'}+k/N]$ and $\widetilde{\mathcal{P}_{l,N}}$ is the union of the set of all such points for $l'\leq l$ with $\mathcal{P}_{0,N}$.

We remark that the set $\widetilde{\mathcal{P}_{l,N}}$ is closed. Indeed, the set of accumulation points are the set of all fractions that can be written as $[n_1,\ldots,n_{l'}]$ for $l'<l$ and these are in $\widetilde{\mathcal{P}_{l,N}}$ by definition. Hence its complement is a countable disjoint collection of open intervals and we can define the following.

\begin{definition}
    For $\alpha \in (0,1)$, we define the operator $P_{l,N}:C^\alpha([0,1])\to C^\alpha([0,1])$ by letting $P_{l,N}(f)(x)= f(x)$ for $x\in \widetilde{\mathcal{P}_{l,N}}$. If $x\neq \widetilde{\mathcal{P}_{l,N}}$, let $a$ and $b$ the respective largest and smallest element of $\widetilde{\mathcal{P}_{l,N}}$ such that $ a<x<b$ and interpolate the function $f$ between the two points to obtain
    \[P_{l,N}(f)(x) =  \frac{b-x}{b-a}f(a) + \frac{x-a}{b-a}f(b).\]
\end{definition}
We prove it is well-defined in Lemma \ref{lemma: well-defined} in the appendix. It is fairly easy to show that $P_{l,N}$ is \textit{not} a compact operator. However, we shall show that $\mathcal{L}_{\beta}^l \circ P_{l,N}$ is compact. As an intermediate step, we prove the following proposition. 
\begin{proposition}\label{proposition: uniformising convergence}
    Let $\alpha \in (0,1)$ and let $l\in \mathbb{N}$. Let $K>0$. For each $k\in \mathbb{N}$, let $\mathcal{F}_k\subset C^\alpha([0,1])$ be a family of functions where all $f\in \mathcal{F}_k$ satisfy $\|f\|\leq K$. Assume furthermore that for all $n_1,n_2,\ldots,n_l\in\mathbb{N}$ we have that
    \begin{equation}\label{equation: holder convergence on the cylinder sets}
        \lim_{k\to \infty}\sup_{f\in \mathcal{F}_k}\|f\circ \psi_{n_1,\ldots,n_l} \|=0.
    \end{equation}
    Then 
    \[\lim_{k\to \infty}\sup_{f\in \mathcal{F}_k}\|L_{\beta}(f)\circ \psi_{n_2,\ldots,n_l}\|=0\]
    for all $n_2,\ldots,n_l$ if $l>1$ and $\lim_{k\to \infty}\sup_{f\in \mathcal{F}_k}\|L_{\beta}(f)\|= 0$ if $l=1$.
\end{proposition}
Applying the above proposition $l$ times, we obtain the following corollary.
\begin{corollary}\label{corollary: operator is compact}
    Let $l\in \mathbb{N}$ and let $\mathcal{F}_k$ be a collection of subsets of $C^\alpha([0,1])$ as in Proposition \ref{proposition: uniformising convergence}. Then $\lim_{k\to \infty}\sup_{f\in \mathcal{F}_k}\|L^l_{\beta}(f)\|= 0$.
\end{corollary}
The corollary is thus a way of converting a statement of the convergence of the family $\mathcal{F}_k$ of functions on individual sets of the form $\{[n_1,\ldots,n_l+x]: x\in [0,1]\}$ to convergence of $L_\beta^l(\mathcal{F}_k)$ with respect to the $\|\cdot\|$-norm. 
\begin{proof}[Proof of Proposition \ref{proposition: uniformising convergence}.]
Our first claim is that \eqref{equation: holder convergence on the cylinder sets} implies that $\lim_{n\to \infty}\sup_{f\in \mathcal{F}_n}\|f\|_\infty=0$. Let $\varepsilon>0$. Let $\mathcal{J}\subset \mathbb{N}^l$ be a finite collection of digits such that the set 
\[A:= \bigcup_{(n_1,\ldots, n_l)\in J}\langle n_1,\ldots,n_l\rangle\]
satisfies $d(x,A)\leq \varepsilon $ for all $x\in [0,1]$. By \eqref{equation: holder convergence on the cylinder sets} , there exists some $N\in\mathbb{N}$ for which $\abs{f(y)}\leq \varepsilon$ for all $y\in A$ and $f\in \mathcal{F}_k$ with $k\geq N$. For any $x\in [0,1]$, there exists a $y\in A$ with $d(y,x)\leq 2\varepsilon$. Hence by uniform boundedness of the Hölder norm,
\[ f(x)\leq \varepsilon + (2\varepsilon)^\alpha K .\]
Letting $\varepsilon\to 0$, we obtain the claim. It suffices therefore to show that $\lim_{k\to \infty}\sup_{f\in \mathcal{F}_k}\|f\circ \psi_{n_2,\ldots,n_l}\|_\alpha=0$. Assume for the moment that $l>1$.

We need uniform estimates on the Hölder seminorm of the function.
\begin{equation}\label{equation: cylinder sets for transfer operator}
    \begin{split}
        L_\beta(f)([n_2,n_3,\ldots,n_l+x])&= f(0)\zeta_H(2\beta,1+[n_2,n_3,\ldots,n_l+x])\\
        &+\sum_{n_1=1}^\infty [n_1,n_2,\ldots,n_l+x]^{2\beta}f^\star\left([n_1,n_2,\ldots,n_l+x]\right),
    \end{split}
\end{equation}
where $f^\star = f-f(0)$. Let $\varepsilon >0$ and note that there exists some $N_1\in \mathbb{N}$ depending only on $K,\beta$ and $\alpha$ such that \[\sup_{x\neq y}\abs{f(0)}\frac{\abs{\zeta_H(2\beta,1+[n_2,n_3,\ldots,n_l+y])-\zeta_H(2\beta,1+[n_2,n_3,\ldots,n_l+x])}}{\abs{x-y}^\alpha}\leq \varepsilon
\]
for all $f\in \mathcal{F}_k$ with $k\geq N_1$, since $\zeta_H$ is analytic and $\sup_{f\in \mathcal{F}_k}\abs{f(0)}\to 0$ as $k\to\infty$. We now estimate the contribution to the Hölder seminorm from the second term in \eqref{equation: cylinder sets for transfer operator}. For $x,y\in[0,1]$, denote $\hat{x}= [n_2,\ldots,n_l+x]$ and $\hat{y}=[n_2,\ldots,n_l+y]$. Let $M\in \mathbb{N}$ such that for all $f\in C^\alpha([0,1])$ with $\|f\|\leq K$:
\[
    \sum_{n=M}^\infty\abs{\frac{f\left(\frac{1}{n+\hat{x}}\right)}{(n+\hat{x})^{2\beta}}-\frac{f\left(\frac{1}{n+\hat{y}}\right)}{(n+\hat{y})^{2\beta}}} \leq \varepsilon \abs{\hat{x}-\hat{y}}^\alpha
\]
for all $x\neq y$. The existence of such an $M$ follows from uniform boundedness of the Hölder norm and the estimate \eqref{equation: summand for Hölder} we used in the proof of Lemma \ref{lemma: mayer operator inequalities}. Furthermore, it is a fairly straightforward exercise to show that \eqref{equation: holder convergence on the cylinder sets} implies that 
\[ \sup_{f\in\mathcal{F}_k}\sum_{n=1}^{M-1}\abs{\frac{f\left(\frac{1}{n+\hat{x}}\right)}{(n+\hat{x})^{2\beta}}-\frac{f\left(\frac{1}{n+\hat{y}}\right)}{(n+\hat{y})^{2\beta}}}\leq \varepsilon \abs{\hat{x}-\hat{y}}^\alpha \]
for all $f\in \mathcal{F}_k$ with $k$ greater than some $N_2\in\mathbb{N}$ and $x,y\in[0,1]$ with $x\neq y$. We obtain by \eqref{equation: cylinder sets for transfer operator} and the above three equations that 
\[\frac{\abs{L(f)([n_2,n_3,\ldots,n_l+y])-L(f)([n_2,n_3,\ldots,n_l+x])} }{\abs{[n_2,n_3,\ldots,n_l+y]-[n_2,n_3,\ldots,n_l+x]}^\alpha}\leq 3\varepsilon\]
for all $f\in \mathcal{F}_k$ with $k>\max\{N_1,N_2\}$. The lemma for $l>1$ follows from letting $\varepsilon\to 0$. For the case $l=1$, replace $[n_2,n_3,\ldots,n_l+x]$ , $[n_2,n_3,\ldots,n_l+y]$ and $[n_1,n_2,\ldots,n_l+x]$ with $x,y$ and $1/(n_1+x)$ respectively. 

\end{proof}

\begin{proposition}\label{proposition: compact operator}
    The operator $L_{\beta}^l \circ P_{l,N}$ is compact for each $l\in\mathbb{N}$ and $N\in\mathbb{N}$.
\end{proposition}
\begin{proof}
    Let $B_{C^\alpha}$ denote the set of functions with $\|f\|<1$. We shall prove $(L_{\beta}^l \circ P_{l,N})(B_{C^\alpha})$ is precompact by showing that any sequence in the set has a Cauchy subsequence. 

    Let $g_n$ be a sequence of functions in $(L_{\beta}^l \circ P_{l,N})(B_{C^\alpha})$ and let $f_n$ be a sequence of functions in $B_{C^\alpha}$ with $(L_{\beta}^l \circ P_{l,N})(f_n)= g_n$ for all $n$. For any $x\in  \widetilde{\mathcal{P}_{l,N}}$ there exists a subsequence $f_{m_k}$ such that $f_{m_k}(x)$ is a Cauchy sequence. By a diagonal argument, we may assume $f_{m_k}(x)$ is Cauchy for all $x\in  \widetilde{\mathcal{P}_{l,N}}$. 

    For any $n_1,\ldots,n_l\in \mathbb{N}$, let $I$ be the closed interval spanned by $[n_1,\ldots,n_l]$ and $[n_1,\ldots,n_l+1]$. The intersection $I\cap \widetilde{\mathcal{P}_{l,N}}$ is finite. Indeed, it consists of the endpoints, the points $[n_1,n_2,\ldots,n_l+k/N]$ with $k\in\{0,\ldots,N-1\}$ and the points $[n_1,n_2,\ldots,n_l'+p/N]$ with  $l'<l$ and $p/N \in \langle n_{l'+1},\ldots,n_l \rangle$.

    Since the functions $P_{l,N}(f_{m_k})$ restricted to $I$ are piecewise linear functions interpolated at a finite number of points $x$ for which $f_{m_k}(x)$ is Cauchy, the functions $\rstr{P_{l,N}(f_{m_k})}{I}$ form a Cauchy sequence in $C^\alpha(I)$. Since $\psi_{n_1,\ldots,n_l}$ is analytic on a neighbourhood of $[0,1]$, it follows that $P_{l,N}(f_{m_k})\circ\psi_{n_1,\ldots,n_l}$ is Cauchy on $C^\alpha([0,1])$. In particular, the families of functions defined by
    \[\mathcal{F}_k = \{P_{l,N}(f_{m_r})-P_{l,N}(f_{m_s}): r,s\geq k\} \]
    satisfy the conditions of Proposition \ref{proposition: uniformising convergence}. By Corollary \ref{corollary: operator is compact}, we thus obtain that
    \[\lim_{k\to \infty}\sup_{r,s\geq k}\left\|(L_{\beta}^l \circ P_{l,N})(f_{m_r}) -(L_{\beta}^l \circ P_{l,N})(f_{m_s})\right\|= 0\]
    which shows $g_{m_k}$ is a Cauchy subsequence.
\end{proof}

Let us now prove most of Theorem \ref{theorem: radius of essential spectrum}. 
\begin{theorem}\label{theorem: spectral theorem}
Let $0<\alpha<1 $. Then for all $\beta$ with $\Re(\beta)>(1-\alpha)/2$ the transfer operator $L_\beta$ satisfies 
\[\rho_e(L_\beta) \leq  \rho(\mathfrak{L}_{\Re(\beta)+\alpha}).\]
\end{theorem}
\begin{proof}
    By Nussbaum's formula and Proposition \ref{proposition: compact operator}, it suffices to prove that there is some constant $C''>0$ for which \[\limsup_{N\to\infty} \|L_{\beta}^l -L_{\beta}^l \circ P_{l,N} \| \leq C''\left\|\mathfrak{L}_{\Re(\beta)+\alpha}^l\right\| \text{ for all } l.\]

    We show in Lemma \ref{lemma: well-defined} in the appendix that there exists a constant $E>$ for which $\|P_{l,N}\|\leq E$ for all $l,N$. 

    For fixed $l$, we first claim that
    \[ \|L_{\beta}^{l'} -L_{\beta}^{l'} \circ P_{l,N} \|_\infty \to 0 \text{ as } n\to \infty \]
    for all $l'\leq l$. Let $f\in C^\alpha([0,1])$. For notational simplicity, let $g_{l,N}:= f- P_{l,N}f$, which by the aforementioned lemma satisfies $\|g_{l,N}\|\leq (E+1)\|f\|$ for all $l,N$. Since $g_{l,N}(0)=0$ we have that 
    \[L_{\beta}(g_{l,N})(x)= \sum_{n=1}^\infty \frac{1}{(n+x)^{2\beta}}(g_{l,N})\left(\frac{1}{n+x}\right).\]
    It follows from the above formula that since $g_{l,N}(x)=0$ for all $x\in \mathcal{P}_{k,N}$ with $k\leq l$, the equality $L_\beta(g_{l,N})(x)=0$ holds for all $x\in \mathcal{P}_{l-1,N}$. By repeating this argument, we obtain that $L_\beta^{l'}(g_{l,N})(x)=0$ for all $x\in \widetilde{\mathcal{P}_{l-l',N}}$. By definition of the $\|\cdot\|_\alpha$-seminorm, we obtain the inequality $\|L_\beta^{l'}(g_{l,N})\|_\infty \leq \lambda^{\alpha}\| L_\beta^{l'}(g_{l,N})\|_\alpha$, where $\lambda$ is the maximal distance of a point $x$ to $\widetilde{\mathcal{P}_{l-l',N}}$. We thus obtain 
    \[\|(L_{\beta}^{l'} -L_{\beta}^{l'} \circ P_{l,N})(f)\|_\infty= \|L_\beta^{l'}(g_{l,N})\|_\infty\leq (E+1)\lambda^{\alpha}\| L_\beta^{l'}\|\|f\| .\]
    Letting $N\to\infty$ proves the claim. 

    To show the theorem use the claim for $l=l'$ and Lemma \ref{lemma: mayer operator inequalities} to obtain for any $\varepsilon>0$ that
    \[\limsup_{N\to\infty }\left\|L_{\beta}^{l} -L_{\beta}^{l} \circ P_{l,N} \right\|\leq \limsup_{N\to\infty }\sup_{\substack{f\in C^{\alpha}([0,1]) \\ \|f\|\leq 1}}\left\|  \mathfrak{L}_{\Re(\beta)+\alpha}\left((1+\varepsilon)\abs{L_{\beta}^{l-1}(g_{l,N})}_\alpha\right) \right\|_\infty .\]
    Since $\varepsilon$ is arbitrary, the above also holds for $\varepsilon=0$. 
    Applying the claim for $l'= l-1,l-2,\ldots,0$ we obtain once again by Lemma \ref{lemma: mayer operator inequalities} that 
    \[\limsup_{N\to\infty }\left\|L_{\beta}^{l} -L_{\beta}^{l} \circ P_{l,N} \right\|\leq\limsup_{N\to\infty }\sup_{\substack{f\in C^{\alpha}([0,1]) \\ \|f\|\leq 1}}\left\|  \mathfrak{L}_{\Re(\beta)+\alpha}^l\left(\abs{g_{l,N}}_\alpha\right) \right\|_\infty  \leq (E+1)\left\| \mathfrak{L}_{\Re(\beta)+\alpha}^l \right\|.\]
    Letting $C''=E+1$ proves the theorem. 
    
    \end{proof}
\begin{remark}\label{remark: spectral radius}
By using the spectral radius formula $\rho(\mathfrak{L}_{\Re(\beta)+\alpha})=\limsup_{l\to\infty}\|\mathfrak{L}_{\Re(\beta)+\alpha}^l\|^{1/l}$ and bounding every function $f\in L^\infty([0,1])$ by a constant, we obtain that $\rho(\mathfrak{L}_{\Re(\beta)+\alpha})\leq \rho(\mathcal{L}_{\Re(\beta)+\alpha})$ and hence that 
\[\rho(\mathfrak{L}_{\Re(\beta)+\alpha}) = \lambda_1(\Re(\beta)+\alpha),\]
where $\lambda_1(t)$ for $t>1/2$ is the eigenvalue of maximum modulus of $\mathcal{L}_{t}$. The eigenvalue is guaranteed to be simple and positive by the Perron-Frobenius theorem. In fact $\lambda_1(t)$ is a real-analytic and decreasing function of $t$ with $\lambda(1)=1$, see \cite{Mayer90,Vallee98}. 
\end{remark}
Theorem \ref{theorem: radius of essential spectrum} for $\alpha \in (0,1)$ then follows from Remark \ref{remark: spectral radius} and from the following corollary of Theorem \ref{theorem: spectral theorem}.
\begin{corollary}\label{Corollary: Ruelle}
    Let $\lambda \in \mathbb{C}$ satisfy $\abs{\lambda}>\rho(\mathfrak{L}_{\Re(\beta)+\alpha})$. Then $\lambda$ is an  eigenvalue of $L_\beta$ if and only if it is an eigenvalue of $\mathcal{L}_\beta$. In this case the respective eigenspaces $E_\lambda^\alpha$ and $E_\lambda^\omega$ are equal. 
\end{corollary}
This is an example of a more general principle for operators on Banach spaces, see e.g. \cite{Grabiner}. For our purposes, we may directly adapt a proof due to Ruelle (Corollary 3.3 in \cite{Ruelle89}), which we provide in the appendix. 

\begin{proof}[Proof of Theorem \ref{theorem: radius of essential spectrum}]

Let $k:=\lfloor \alpha \rfloor$ and let $\eta = \alpha -\lfloor \alpha \rfloor$. We first assume $\eta>0$ and adapt a standard trick (see e.g. \cite{ColletIsola} III) to reduce to the problem on $C^\eta([0,1])$. We note that for $\Re(\beta)>>1$ we obtain by first applying the product rule for higher derivatives and then the chain rule that there exists coefficients $c_r$ with $c_0=1$ for which 
\begin{equation*}
    \begin{split}
        \frac{d^k L_\beta(f)}{dz^k}(z) =
         \sum_{r=0}^k c_r L_{\beta+r/2+(k-r)}\left(\frac{d^{k-r}f}{dz^{k-r}}\right)(z).
    \end{split}
\end{equation*}
This holds using analytic continuation to $\Re(\beta)>(1-k)/2$. 
Recall that the any composition of operators is compact if one of them is compact and that the operator $C^\alpha([0,1]) \to C^\eta([0,1])$ which sends $f$ to $d^{k-r}f/dz^{k-r}$ is compact for $r>0$. Hence we can write 
\[\frac{d^k L_\beta(f)}{dz^k} = L_{\beta+k}\left(\frac{d^k L_\beta(f)}{dz^k}\right) +\mathcal{K}(f),\]
where $\mathcal{K}$ is some compact operator from $C^\alpha([0,1]) \to C^\eta([0,1])$. Since the essential spectral radius does not depend on compact perturbations, we have that the essential spectral radius of $L_{\beta}$ as an operator on $C^\alpha([0,1])$ is the same as the essential spectral radius of $L_{\beta+k}$ as an operator on $C^\eta([0,1])$. This follows from \eqref{equation: Nussbaum}, but is perhaps easier to see using the Banach space isomorphism $C^\alpha([0,1])\to \mathbb{R}\times C^\eta([0,1]): f\mapsto \left(\left(\frac{d^kf}{dz^r}(0)\right)_{r=0}^{k-1},\frac{d^kf}{dz^k}\right) $. The case $\eta=0$ follows \textit{mutatis mutandis}, replacing $C^\eta([0,1])$ with $C([0,1])$. Hence Theorem \ref{theorem: radius of essential spectrum} for $\alpha>1$ follows from the case $\alpha\in (0,1) $.
\end{proof}

\begin{remark}\label{remark: abstract}
    The zeroes of $Z(s)$ on the line $s=1/2+it$ are those for which there are eigenvalues of $\Delta$ called \enquote{Maass cusp forms} with eigenvalue $s(1-s)$. Furthermore the multiplicity of the zeroes and the eigenspaces of $\Delta$ match. The other zeroes of $Z(s)$ are a simple zero at $s=1$ and at the values of $s$ for which $2s$ is a nontrivial zero of the Riemann zeta function $\zeta$. Since $\rho(\mathcal{L}_a)<1$ when $a>1$, we see that the eigenspaces of $\mathcal{L}_\beta$ and $L_\beta$ for the eigenvalues $\lambda=\pm 1$ are identical when $1>\alpha>1-\Re(\beta)$. In particular, the eigenspaces associated to the Maass cusp forms are preserved when $\alpha >1/2$, and the eigenspaces associated to the nontrivial zeroes of the Riemann zeta function\footnote{In fact, the zeroes of $Z(s)$ associated with the nontrivial zeroes of $\zeta$ always occur due to $1$ being an eigenvalue of $\mathcal{L}_\beta$.} to the right of the critical line are preserved when $\alpha=3/4$.
\end{remark}

We conclude with a fairly straightforward remark which will make the next section easier from a notational perspective. 
\begin{remark}\label{remark: extending L_beta}
    We can of course consider the operator $L_\beta$ acting on spaces other than $A_\infty(D)$ or $C^\alpha([0,1])$. For example, it is clear that the estimates on the essential spectrum still hold for $L_\beta$ acting on $C^\alpha([a,M])$ with $0<a\leq 1$ and $M\geq(1+a)^{-1}$. As we will remark upon Section \ref{section: three-term functional equation}, eigenfunctions of $\mathcal{L}_\beta$ extend to functions on $\mathbb{C}\setminus(-1,\infty)$, so the point spectrum beyond the estimate of Theorem \ref{theorem: radius of essential spectrum} does not depend on $a,M$. In fact by letting $a\to 0$ and $M\to \infty$, we see that a natural space for $L_\beta$ is the space $C^\alpha((-1,\infty))$, which we define to be the Fréchet space of functions $f$ such that on every compact $K\subset (-1,\infty)$ we have that $f$ is $\lfloor \alpha\rfloor$-times differentiable such that the $\lfloor \alpha \rfloor$-th derivative is continuous if $\alpha=\lfloor \alpha\rfloor$ and is $(\alpha-\lfloor \alpha\rfloor)$-Hölder continuous if not. 
\end{remark}

\section{Relation with the three-term functional equation.}\label{section: three-term functional equation}
In this section, we consider the following functional equation.
\begin{equation}\label{equation: three term functional}
     \lambda f(z)-\lambda f(z+1) = \left(\frac{1}{z+1}\right)^{2\beta}f\left(\frac{1}{z+1}\right),
\end{equation}
where $\lambda\in \mathbb{C}\backslash \{0\}$. This equation was first introduced (in slightly modified form) for $\lambda=1$ in \cite{Lewis97} and for $\lambda =\pm 1$ in \cite{ZagierLewis}. The formulation in this document can be found in \cite{MayerChang96}. In particular, we consider this functional equation for functions on $(-1,\infty)$ of varying regularity and for holomorphic functions on $\mathbb{C}\backslash(-\infty,1]$.

It is fairly easy to prove that any eigenfunction of $\mathcal{L}_\beta$ with eigenvalue $\lambda\neq 0$ satisfies \eqref{equation: three term functional}. Indeed, this follows from the identity $\zeta(s,z)-\zeta(s,z+1)=z^{-s}$, see e.g. Proposition 7 in \cite{MayerChang96}. Furthermore, from the identity 
\[f= \lambda^{-1}\left(\sum_{n=0}^k\frac{d^nf}{dz^n}(0)\frac{\zeta_H(n+2\beta,z+1)}{n!}+\sum_{n=1}^\infty \frac{1}{(n+z)^{2\beta}}(f^\star)\left(\frac{1}{n+z}\right), \right),\]
we can use the contraction properties of the inverse branches $\psi_n$ inductively to extend $f$ to a function which is holomorphic on $\mathbb{C}\setminus (-\infty, 1]$.

Conversely, the space of solutions to \eqref{equation: three term functional} is uncountable-dimensional, something already remarked upon in \cite{ZagierLewis}. However, we obtain a one-to-one correspondence after imposing a condition on the asymptotic behaviour of $f$. Let us first describe the asymptotic behaviour of a generic solution in the non-analytic case (e.g. at Hölder regularity). 

Define $C^\alpha(\mathbb{R}/\mathbb{Z})$ be the space of all $\lfloor \alpha \rfloor$-times continuously differentiable functions on the torus with $(\alpha-\lfloor \alpha\rfloor)$-Hölder derivative if $\alpha$ is not an integer. We identify these functions with $1$-periodic functions on $C^\alpha((-1,\infty))$ by setting $Q(z)=Q(z\mod 1)$ for $z\in (-1,\infty)$.
\begin{proposition}\label{equation: proposition: }
   Let $f\in C^\alpha((-1,\infty))$ solve \eqref{equation: three term functional} for some $\lambda \neq 0$. Assume furthermore that $\Re(\beta)>(1-\alpha)/2$ and $\beta\notin \{1/2,0,-1/2,-3/2,\ldots\}$. Then there exist a function $Q\in C^\alpha((-1,\infty))$ and coefficients $C_0,\ldots,C_{\lfloor\alpha\rfloor},C_0^*,\ldots C_{\lfloor\alpha\rfloor}^* $ depending only on $\lambda,\beta$ and the derivatives $\left\{f(0),\frac{df}{dz}(0),\ldots \frac{d^{\lfloor \alpha\rfloor}f}{dz^{\lfloor \alpha\rfloor}}(0)\right\}$ such that 
   \begin{equation}\label{equation: Qinfinity}
       f(z)= Q(z)+ \sum_{n=0}^{\lfloor \alpha\rfloor} C_n z^{{1-2\beta-n}}+O(x^{1-2\Re(\beta)-\alpha}) \text{ as } z\to\infty
   \end{equation}
   and 
      \begin{equation}\label{equation: Q0}
       f(1-z)=  \frac{1}{\lambda z^{2\beta}}Q(1/z)+ \sum_{n=0}^{\lfloor \alpha\rfloor} C_n^* z^{{n-1}}+O(x^{\alpha-1}) \text{ as } z\to 0,
   \end{equation}
   where we may replace the $O(x^{1-2\Re(\beta)-\alpha}),O(x^{\alpha-1})$ with $o(x^{1-2\Re(\beta)-\alpha}),o(x^{\alpha-1})$ respectively if $\alpha=\lfloor \alpha \rfloor $. 
\end{proposition}
\begin{proof}
    This is essentially the proof of the proposition in Chapter III section 3 of \cite{ZagierLewis} adapted to the non-analytic case and including the case $\lambda \neq \pm 1$.
    
    Define the function $Q:= f-\lambda^{-1}L_\beta(f)$, where we extend $L_\beta$ to an operator on $C^\alpha((-1,\infty))$ as in Remark \ref{remark: extending L_beta}. For $\Re(\beta)>1$ it is clear that 
    \[ Q(z)-Q(z+1)= f(z)-f(z+1)- \lambda^{-1}\frac{1}{\left(1+z\right)^{2\beta}}f\left(\frac{1}{1+z}\right)=0\]
    and this formula continues to hold for all $\Re(\beta)> (1-\alpha)/2$ (with $\beta \neq 1/2$) by analytic continuation. It is also clear from the definition that $Q\in C^\alpha(\mathbb{R}/\mathbb{Z})$.

    Assume $\alpha=k+\eta$ with $(k,\eta)\in \mathbb{N}\times[0,1)$ We now consider the asymptotic behaviour of $L_\beta(f)(z)$. We note 
    that if $\eta >0$ and $z\in [0,1]$,
    \begin{equation}\label{equation: f expansion}
        \abs{f(z)-\sum_{n=0}^k\frac{d^nf}{dz^n}(0)\frac{z^n}{n!}}\leq \left\|\frac{d^kf}{dz^k}\right\|_\eta \frac{z^\alpha}{k!},
    \end{equation}
    where $\|g\|_\eta$ is the $\eta$-Hölder coefficient of $g$ on $[0,1]$. Hence we obtain by an application of the integral test that if $f^\star(z)=f(z)-\sum_{n=0}^k\frac{d^nf}{dz^n}(0)\frac{z^n}{n!}$,  \[\sum_{n=1}^\infty \frac{1}{(n+z)^{2\beta}}(f^\star)\left(\frac{1}{n+z}\right)= O(z^{1-2\Re(\beta)-\alpha})\]
    as $z\to \infty $
    with the implied constant easily computable and depending on $\left\|\frac{d^kf}{dz^k}\right\|_\eta, \beta$ and $\alpha$. We therefore obtain by \eqref{equation: Mayer transfer extended2} 
    that \begin{equation}\label{equation: behavior of L}
     L_\beta(f)(z)= \sum_{n=0}^k\frac{d^nf}{dz^n}(0)\frac{\zeta_H(n+2\beta,z+1)}{n!}+ O(z^{1-2\Re(\beta)-\alpha}) \text{ as } z\to \infty   
    \end{equation}
    We obtain \eqref{equation: Qinfinity} for $\eta>0$ by writing $f=Q +\lambda^{-1}L_\beta(f)$ using the identity $\zeta_H(s,z)=\zeta_H(s,z+1)+z^{-s}$ and the well-known expansion (see \cite{DLMF} \S 25.11) 
    \begin{equation}\label{equation: Hurwitz asymptotics}
        \zeta_H(s,z)\sim\frac{z^{1-s}}{s-1}+\frac{1}{2}z^{-s}+ \sum_{k=1}^\infty \frac{B_{2k}}{2k!}\frac{\Gamma(s+2k-1)}{\Gamma(s)}z^{1-s-2k} \text{ as }z\to \infty,
    \end{equation}
    where $B_{2k}$ are the Bernoulli numbers.

    For $\eta=0$, we have that the left hand side of \eqref{equation: f expansion} is $o(1)$ as $z\to 0$ and we follow our previous argument, mutatis mutandis. The asymptotic behavior of $f(1-z)$ for small $z$ follows from the identity $f(1-z)= \frac{\lambda^{-1}}{z^{2\beta}}f(1/z)+f(z)$, using the asymptotics \eqref{equation: Qinfinity} to obtain the expansion for $f(1/z)$ as $z\to 0$ and the expansion $ f(z)= \sum_{n=0}^{\lfloor\alpha\rfloor} \frac{d^nf}{dz^n}(0)\frac{z^k}{k!}+O(z^{\alpha})$.

\end{proof}
We see by Proposition $4.1$ that a generic solution of \eqref{equation: three term functional} in $C^\alpha((-1,\infty))$ satisfies $f= O(z^{1-2\Re(\beta)})$ if $\Re(\beta)>(1-\alpha)/2$. On the other hand, we see that the condition $Q=0$  in \eqref{equation: Qinfinity} or in \eqref{equation: Q0} is equivalent to the condition $L_\beta f= \lambda f$, which by Theorem \ref{theorem: radius of essential spectrum} implies $f$ extends to a holomorphic function in $\mathbb{C}\setminus (-\infty,-1]$ if $\abs{\lambda}\leq \rho(\mathcal{L}_{\Re(\beta)+\alpha})$. For example, we have the following corollary of Proposition \ref{equation: proposition: }, which we can view as a stronger version of the Bootstrapping procedure by Lewis and Zagier in \cite{ZagierLewis}.
\begin{corollary}
    Let $\Re(\beta)=1/2$. Suppose $f:(-1,\infty)\to \mathbb{C}$ is some solution of \eqref{equation: three term functional} for $\abs{\lambda} = 1$. Assume furthermore it satisfies the mild growth condition $f(z)= C_0x^{1-2\beta}+o(1)$ as $z\to \infty$.
    Then if $f$ is $1/2+\varepsilon$-Hölder for some $\varepsilon>0$, it is automatically holomorphically extendable to the cut plane $\mathbb{C}\setminus (-\infty,-1]$.
\end{corollary}
This corollary immediately follows from Theorem \ref{theorem: radius of essential spectrum} and the fact that $\rho(\mathcal{L}_a)<1$ if $a>1$. We remark that when $\Re(\beta)=1/2$ and $\lambda =1$ or $-1$, the space of eigenvalues of $\mathcal{L}_{\beta}$ is in bijection with the space of even, respectively uneven Maass cusp forms on the modular surface with eigenvalue $\beta(1-\beta)$ and we have that $C_0=0$, i.e. $f(z)= o(1)$ as $z\to \infty$, see \cite{Lewis97, ZagierLewis}.

Proposition \ref{equation: proposition: } showed that a solution of the three term equation has the associated periodic function $(1-L_\beta)f$. By applying the inverse, we obtain the following proposition.

\begin{proposition}\label{Proposition: Constructing Q}
    Let $\alpha>0$, $\lambda\in\mathbb{C}\setminus\{0\}$ and let $Q\in C^\alpha(\mathbb{R}/\mathbb{Z})$.
    For all $\beta$ for which $\Re(\beta)> (1-\alpha)/2$ and \[r(\mathcal{L}_{\Re(\beta)+\alpha})< \abs{\lambda},\]
    the function
    \[f=(1 - \lambda^{-1}L_\beta)^{-1}Q,\]
    is a meromorphic expression over all $\beta$ satisfying the conditions in this proposition. Furthermore $f$ satisfies \eqref{equation: three term functional}, \eqref{equation: Qinfinity} and \eqref{equation: Q0}.
\end{proposition}
\begin{proof}
    We interpret $(\lambda - L_\beta)^{-1}$ in the sense of Remark \ref{remark: extending L_beta}. More specifically, for $L_\beta$ acting on $ C^r([a,(1-a)^{-1}])$, the expression $(1 - \lambda^{-1}L_\beta)^{-1}\left(\rstr{Q}{[a,(1-a)^{-1}]}\right)$ is meromorphic and satisfies 
\begin{equation}\label{equation: expansion for resolvent}
    \begin{split}
        &(1 - \lambda^{-1}L_\beta)^{-1}\left(\rstr{Q}{[a,(1-a)^{-1}]}\right)(z)=\sum_{n=0}^\infty  \lambda^{-n}L_\beta^n\left(\rstr{Q}{[a,(1-a)^{-1}]}\right)(z).\\
        &= Q(z)+\sum_{n=1}^\infty \lambda^{-n}\sum_{a_1,\ldots,a_n\in\mathbb{N}} Q([a_1,\ldots,a_n+z])\Pi_{k=1}^n[a_k,\ldots,a_n+z]^{2\beta}
    \end{split}
\end{equation}
if $\Re(\beta)$ is large enough. By Theorem \ref{theorem: radius of essential spectrum}, we see that $(1 - \lambda^{-1}L_\beta)^{-1}$ is a meromorphic operator-valued function for all $a\leq1$. Hence the functions $(1 - \lambda^{-1}L_\beta)^{-1}\left(\rstr{Q}{[a,(1-a)^{-1}]}\right)$ are 
meromorphic functions over the half-plane consisting of all $\beta$ satisfying the conditions in the proposition. We see by the second line of \eqref{equation: expansion for resolvent} and uniqueness of analytic continuation that these functions agree for different $a$ on the intersection of their domains. Hence we may define $(1 - \lambda^{-1}L_\beta)^{-1}Q(z)$ for any $z\in (-1,\infty)$ by choosing $a$ in \eqref{equation: expansion for resolvent} small enough.

We may thus define
\[ f:=(1 - \lambda^{-1}L_\beta)^{-1}Q(z)=Q(z)+\sum_{n=1}^\infty \lambda^{-n}\sum_{a_1,\ldots,a_n\in\mathbb{N}} Q([a_1,\ldots,a_n+z])\Pi_{k=1}^n[a_k,\ldots,a_n+z]^{2\beta}\]
For $\Re(\beta)$ large enough and use \eqref{equation: expansion for resolvent} to obtain the analytic continuation. It is clear from the above formula that $f$ satisfies \eqref{equation: three term functional} for large $\Re(\beta)$ and hence by analytic continuation for all $\beta$ satisfying the condition in the proposition. 

The proposition follows immediately from the proof of Proposition \ref{equation: proposition: }.
\end{proof}

From the discussion so far we obtain the following theorem in a straightforward manner. 
\begin{theorem}
    Let $\beta\notin \{1/2,0,-1/2,\ldots\}, \lambda \neq 0$ and let $\fe^{\lambda,\beta,\alpha}$ denote the set of solutions of \eqref{equation: three term functional} belonging to $C^\alpha((-1,\infty))$. Assume furthermore that $\Re(\beta)>(1-\alpha)/2$ and that $\rho(\mathcal{L}_{\Re(\beta)+\alpha})<\lambda$. Then the map 
    \[\mathcal{A}^{\lambda,\beta,\alpha}:\fe^{\lambda,\beta,\alpha} \to  C^\alpha(\mathbb{R}/\mathbb{Z})\]
    which sends $f\in \fe^{\lambda,\beta,\alpha}$ to its associated periodic function $Q$ defined in Proposition \ref{equation: proposition: } is bijective if and only if $\lambda\neq \sigma(\mathcal{L}_\beta)$.
\end{theorem}
\begin{proof}
    The map $\mathcal{A}^{\lambda,\beta,\alpha}$ is explicitly given by the restriction of  $f\mapsto (1-\lambda^{-1}L_\beta)f$ to $\fe^{\lambda,\beta,\alpha}$, so theorem follows immediately from Proposition \ref{equation: proposition: } and Proposition \ref{Proposition: Constructing Q}. 
\end{proof}
\begin{remark}
     Proposition \ref{Proposition: Constructing Q} does not give us a construction of solutions of \eqref{equation: three term functional} when $\lambda$ is an eigenvalue of $L_\beta$. We remark that in the case $\lambda =\pm 1$, explicit examples are constructed in \cite{ZagierLewis}, showing that the spaces $\fe^{\pm 1,\beta,\alpha}$ are still uncountable-dimensional. For example, the function $z\mapsto Q(z)+(z+1)^{2\beta -1}Q(-1/z)$, where $Q(z)=\mp Q(-z)$ and is $1$-periodic. These constructions seem to rely on an additional symmetry for solutions of \eqref{equation: three term functional} with $\lambda= \pm 1$. We conjecture these might not necessarily exist for other values of $\lambda\in \sigma(L_\beta)$. 
\end{remark}

\section{Further avenues of exploration.}
This preprint provides estimates for the essential spectral radius for the mayer transfer operator, but the technique should be generalizable to a large class of transfer operators associated to the geodesic flow on noncompact geometrically finite Fuchsian groups. We refer to e.g. \cite{MayerChang01,Morita,Pohl16, PohlFedosova20}. One could also think about applying a similar approach in higher dimensions, i.e. for families transfer operators associated to multidimensional continued fraction algorithms or other expanding maps on hypercubes. However, the significance behind the analytic continuation of these operators is less clear. 

\appendix
\section{Appendix}
\subsection{Hölder Estimates}
We begin this appendix with a lemma that will be useful to estimate Hölder norms. 
\begin{lemma}\label{lemma: hölder lemma}
Let $0<\alpha<1$. Let $a<b\leq c<d$ and let $g:\{a,b,c,d\}\to \mathbb{R}$ be a function. Then there exists a constant $E>0$ independent of $a,b,c,d$ for which 
\[\frac{\abs{g(d)-g(a)}}{(d-a)^\alpha}\leq E \max \left\{\frac{\abs{f(b)-f(a)}}{(b-a)^\alpha},  \frac{\abs{f(c)-f(b)}}{(c-b)^\alpha}, \frac{\abs{f(d)-f(c)}}{(d-c)^\alpha}\right\} \]
\end{lemma}
\begin{proof}
We first start with the following basic inequality. Let $\lambda_1,\lambda_2,\lambda_3 >0$. Then
\[\frac{\sqrt{\lambda_1^{2\alpha}+\lambda_2^{2\alpha}+\lambda_3^{2\alpha}}}{(\lambda_1+\lambda_2+\lambda_3)^\alpha}\leq \max\{3^{1/2-\alpha},1\}.\]
This can be proved by noting that if $\alpha <1/2$ the left hand side is maximised for $\lambda_1=\lambda_2=\lambda_3$ and if $\alpha >1/2$ then it is maximised when two of the parameters are zero.
    Assume first that $b\neq c$. Then we can write $(g(d)-g(a))/(d-a)^\alpha$ as the dot product
    \begin{equation*}
        \begin{split}
             \frac{g(d)-g(a)}{(d-a)^\alpha} &= \frac{1}{(d-a)^\alpha}\begin{bmatrix}
                 (b-a)^\alpha & (c-b)^\alpha & (d-c)^\alpha
             \end{bmatrix}\begin{bmatrix}
                 M_1 \\ M_2\\ M_3
             \end{bmatrix} \text{, where} 
            \\ M_1&= \frac{f(b)-f(a)}{(b-a)^\alpha},M_2= \frac{f(c)-f(b)}{(c-b)^\alpha}, M_3= \frac{f(d)-f(c)}{(d-c)^\alpha}.
        \end{split}
    \end{equation*}
    In particular, by submultiplicativity of the dot product with respect to the usual vector norm and the aforementioned inequality we have
    \[\abs{\frac{g(d)-g(a)}{(d-a)^\alpha}}\leq \max\{3^{1/2-\alpha},1\} \sqrt{M_1^2+M_2^2+M_3^3}\leq \max\{3^{1-\alpha},\sqrt{3}\}\max\{\abs{M_1},\abs{M_2},\abs{M_3}\}, \]
    which proves the lemma for $b<c$. For $b=c$ the proof is analogous . 
\end{proof}

\begin{lemma}\label{lemma: well-defined}
    The operator $P_{l,N}:C^\alpha([0,1])\to C^\alpha([0,1])$ is well-defined, continuous and has uniformly bounded operator norm over all $l,N$. 
\end{lemma}
\begin{proof}
    It is fairly straightforward to show that for $f\in C^\alpha([0,1])$ the assignment $f\mapsto P_{l,N}(f)$ is linear. We therefore just need to show continuity and Hölder continuity of $f\mapsto P_{l,N}(f)$ and bound the operator norm. 

    We note that if $(a,b)\subset [0,1]\setminus \widetilde{\mathcal{P}_{l,N}}$ is a connected component, then $P_{l,N}(f)((a,b))\subset f((a,b))$ by the intermediate value theorem, which proves continuity at the accumulation points and hence continuity of $f$ everywhere. 

    We now estimate the Hölder norm. Let $a<d$ be two points in $[0,1]$. We estimate the quantity $H:=\abs{P_{l,N}(f)(d)-P_{l,N}(f)(a)}/{\abs{d-a}^\alpha}$. By perturbing the two points by an arbitrarily small amount, we may assume $a$ and $d$ are not accumulation points of $\mathcal{P}_{l,N}(f)$. Let $(x,b)$ and $(c,y)$ be the connected components of $[0,1]\setminus \widetilde{\mathcal{P}_{l,N}}$ such that $a\in (x,b)$ and $d\in(c,y)$. In case the two intervals are the same, it is clear that $H\leq \|f\|_\alpha$. In case they are different, we let $g=P_{l,N}(f)$ on $\{a,b,c,d\}$ as in Lemma \ref{lemma: hölder lemma}, whence we obtain that 
\begin{equation*}
        H\leq \max\{3^{1-\alpha},\sqrt{3}\}\max\left\{\abs{\frac{g(b)-g(a)}{(b-a)^\alpha}},\abs{ \frac{g(c)-g(b)}{(c-b)^\alpha}},\abs{\frac{g(d)-g(c)}{(d-c)^\alpha}}\right\}.
\end{equation*}
    The lemma follows after noting that $P_{l,N}(f)$ coincides with $f$ on $x,b,c,y$ and using that
\[ \frac{\abs{P_{l,N}(f)(b)-P_{l,N}(f)(a)}}{(b-a)^\alpha} \leq \frac{\abs{f(b)-f(x)}}{(b-x)^\alpha} \text{ and } \frac{\abs{P_{l,N}(f)(d)-P_{l,N}(f)(c)}}{(d-c)^\alpha} \leq \frac{\abs{f(y)-f(c)}}{(y-c)^\alpha}.\]
\end{proof}

\subsection{Ruelle's Corollary}
If $T:B_1\to B_2 $ is a bounded linear operator between Banach spaces, denote its kernel by $\ker(T)$. If $T$ has closed image $\Ima(T)$, denote its cokernel by $\coker(T):= B_2/\Ima(T)$. We say $T$ is Fredholm of index $k$ if $T$ has closed image, $\coker(T):= B_2/\Ima(T)$ and $\ker(T)$ are finite and $\mathrm{dim}\ker(T)-\mathrm{dim} \coker(T) =k $. The composition of two Fredholm operators with respective indices $k,l$ is Fredholm with index $k+l$. 
For a Banach space $B$, we denote its dual by $B^*$ and denote by $T^*$ the induced adjoint operator $T^*:B_2^*\to B_1^*:G\mapsto G\circ T$. The following is a straightforward exercise in functional analysis. 
\begin{lemma}\label{lemma: basic functional analysis}
    If $T:B_1\to B_2 $ has closed image and $\ker(T^*):=\{H\in B_2^*: H\circ T =0\}$ is finite-dimensional, then $\mathrm{dim}\ker(T^*)= \mathrm{dim}\coker(T) $. 
\end{lemma}

Suppose now $B=B_1=B_2$, so $T:B\to B$. We recall the following facts about the spectral theory of operators on Banach spaces
\begin{itemize}
    \item If $\lambda \neq \sigma_e(T)$, then $\lambda I-T$ is Fredholm of index $0$. 
    \item If $T$ is a compact operator, so is $T^*$. 
\end{itemize}
We refer the reader to \cite{Evans18} for more details on the spectral theory of Banach operators, noting that $\sigma_e(F)$ in this preprint corresponds to $\sigma_{e5}(F)$ in their book. 
\begin{proof}[Proof of Corollary \ref{Corollary: Ruelle}.]
    We remind the reader that this proof is a direct adaptation from \cite{Ruelle89}. 
    It suffices to prove that any eigenvalue $\lambda$ of $L_\beta$ with $\lambda > \rho(L_{\Re(\beta)+\alpha})\geq\rho_e(L_\beta)$ is an eigenvalue of $\mathcal{\lambda}_\beta$ and that the dimensions of the respective generalised eigenspaces $E_{L_\beta}(\lambda)=\bigcup_{r=1}^{+\infty} \ker(L_\beta-\lambda I)^r$ and $E_{\mathcal{L}_\beta}(\lambda)=\bigcup_{r=1}^{+\infty}\ker(\mathcal{L}_\beta-\lambda I)^r$ are equal. We note that the unions are over a finite number of nontrivial kernels by definition of the essential spectral radius. By Lemma \ref{lemma: basic functional analysis} and the fact that $\mathcal{L}_\beta^*$ is compact operator and that the relevant operators are Fredholm of index 0, it suffices to prove that the generalised eigenspaces $E_{L_\beta}(\lambda)^*$ and $E_{\mathcal{L}_\beta}(\lambda)^*$ belonging respectively to  $L_\beta^*$ and $\mathcal{L}_\beta^*$ have equal dimension.

    There is a natural (injective) inclusion map $\iota: A_\infty(D) \to C^{\alpha}([0,1])$ defined by restriction to $[0,1]$. The dual operator $\iota^*:{C^{\alpha}}^* \to A_\infty(D)^*$ is defined by restricting distributions on ${C^{\alpha}([0,1])}$ to distributions on $A_\infty(D)$. This restricts to a natural linear map from $E_{L_\beta}(\lambda)^*$ to $E_{\mathcal{L}_\beta}(\lambda)^*$. Surjectivity of this map follows from the Hahn-Banach theorem. We now show injectivity. Suppose that for some $\sigma\in E_{L_\beta}(\lambda)^*$, its restriction $\sigma^A: A_\infty(D) \to \mathbb{C}$ is identically zero. We show $\sigma =0$. 
    
    In order to do so, let $\alpha'<\alpha $ such that $\lambda>\rho(L_{\Re(\beta)+\alpha'})$. Use Hahn-Banach to extend $\sigma$ to $\sigma^{\alpha'}: C^{\alpha'}([0,1])\to\mathbb{C}$. For any $f\in C^\alpha([0,1])$, let $f_n$ be a series of smooth approximations which converge in the $C^{\alpha'}([0,1])$-norm. By cutting of fourier expansions outside of balls of increasing radii in Fourier space, we may assume the maps $f_n$ are analytically extendable to elements of $A_\infty(D)$. Hence we obtain 
    \[0= \lim_{n\to \infty}\sigma^{A}(f_n)=\lim_{n\to \infty}\sigma^{\alpha'}(f_n)= \sigma^{\alpha'}(f)= \sigma(f),\]
    which is what we had to show. 
\end{proof}

\section{Acknowledgements}

\printbibliography
\end{document}